\documentclass{article}

\usepackage{tikz-cd}
\usepackage{comment}
\usepackage[all]{xy}
\usepackage[usenames,dvipsnames]{pstricks}
\usepackage{epsfig}
\usepackage{pst-grad} 
\usepackage{pst-plot} 
\usepackage{ebezier}
\usepackage{latexsym}\usepackage{amsfonts}\usepackage{amssymb}\usepackage{amsthm}
\usepackage{amsmath}\usepackage{newlfont}\usepackage{graphicx}\usepackage{doc}
\usepackage{mathrsfs}
\usepackage[
        colorlinks, citecolor=red,
        backref,
        pdfauthor={MDJ,JK},
        pdftitle={S1}
]{hyperref}

\newtheorem{thrm}{Theorem}
\newtheorem{lem}[thrm]{Lemma}
\newtheorem{cor}[thrm]{Corollary}
\newtheorem{defn}[thrm]{Definition}

\newtheorem{prop}[thrm]{Proposition}
\newtheorem{exam}[thrm]{Example}
\newtheorem*{theorem-non}{Theorem}

\def \Dj{\mbox{\raise0.3ex\hbox{-}\kern-0.4em D}}

\begin{document}

\title{Polynomial entropy on the $n$-fold symmetric product and its suspension}

\author{Ma\v{s}a \Dj ori\'c\\
Matemati\v{c}ki institut SANU\\
Knez Mihailova 36\\
11000 Beograd\\
Serbia\\
masha@mi.sanu.ac.rs\\
\href{https://orcid.org/0000-0003-1364-2339}{https://orcid.org/0000-0003-1364-2339}
}

\maketitle

\begin{abstract} We prove that the polynomial entropy of the induced map $F_n(f)$ on the $n$-fold symmetric product of a compact space $X$ and its suspension are both equal to $nh_{pol}(f)$, when $f:X\to X$ is a homeomorphism with a finite chain recurrent set $\mathcal{CR}(f)$. We also give a lower bound for the polynomial entropy on the suspension, for a homeomorphism $f$ with at least one wandering point, under certain assumptions.
\end{abstract}

\medskip


\section{Introduction}

Every continuous map $f:X\to X$ on a compact metric space $X$ induces a continuous map (called the induced map) on the hyperspace $2^X$ of all nonempty closed subsets. Its subspace, namely the hyperspace $F_n(X)$ of all nonempty subsets with at most $n$ points (for all $n\in\mathbb{N}$), is of main interest in our paper. Furthermore, Barragán defined the $n$-fold symmetric product suspension in \cite{FB}, $n\geqslant 2$, as the quotient space $$\mathcal{S}F_n(X)=F_n(X)/F_1(X).$$ They show that $\mathcal{S}F_n(X)$ is a compact metric space and define the induced map $\mathcal{S}F_n(f)$ which is also continuous.

In the last decades a series of results was obtained that show possible relations between the given (individual) dynamics on $X$ and the induced one (collective dynamics) on the hyperspaces. Without attempting to give complete references, we mention just a few: Borsuk and Ulam~\cite{BU}, Bauer and Sigmund~\cite{BS}, Rom\'an-Flores~\cite{RF}, Banks~\cite{B}, Acosta, Illanes and M\'endez-Lango~\cite{AIM}. In \cite{BST}, Barragán et al. studied the dynamical properties of the dynamical system $(\mathcal{S}F_n(X),\mathcal{S}F_n(f))$, and more generally $(\mathcal{S}F^{m}_n(X),\mathcal{S}F^{m}_n(f))$, in \cite{BST2}.

Kwietniak and Oprocha in~\cite{KO}, Lampart and Raith~\cite{LR}, Hern\'andez and M\'endez~\cite{HM}, Arbieto and Bohorquez~\cite{AB}, and others studied the topological entropy of the induced maps.

The topological entropy is a classical measure of complexity of dynamical systems; one that measures the average exponential growth of the number of distinguishable orbit segments. For dynamical systems with vanishing topological entropy, this growth rate is subexponential, so it is justified to measure it at the polynomial scale. The notion of the polynomial entropy was first introduced by Kurka in~\cite{K} for open covers. Marco defined it for dynamical systems in metric spaces in the context of Hamiltonian integrable systems (see~\cite{M1,M2}). It has been further studied in various contexts by Labrousse~\cite{L1,L2,L3}, Labrousse and Marco~\cite{LM}, Bernard and Labrousse~\cite{BL}, Artigue, Carrasco--Olivera and Monteverde~\cite{AOM}, Haseaux and Le Roux~\cite{HL}, Roth, Roth and Snoha~\cite{RRS}, Correa and de Paula~\cite{CP}, etc.

Some dynamical properties of the topological entropy remain the same for the polynomial entropy, such as the conjugacy invariance, the finite union property and the product formula. However, some of them differ, for example the power formula, the $\sigma$-union property, the variational principle (see~\cite{M1,M2}). In addition, the topological entropy of a system is equal to the topological entropy of the same map, restricted to the non-wandering set. Unlike this, the set of wandering points is visible for the polynomial entropy. For example, Labrousse showed in~\cite{L3} that if a homeomorphism of a compact metric space possesses a wandering point, then its polynomial entropy is greater or equal to $1$ (see Proposition $2.1$ in~\cite{L3}). Moreover, as  Hauseux and Le Roux pointed out in~\cite{HL}, the polynomial entropy is particularly well adapted to the wandering part, since the growth of wandering orbits is always polynomial (see the remarks after Proposition 2.3 in~\cite{HL}).

In~\cite{DK}, the authors computed the polynomial entropy of the induced maps $F_n(f)$ and $C(f)$ when $f$ is a homeomorphism of a circle or an interval with finitely many non-wandering points. Those results heavily rely on the coding techniques developed in~\cite{HL} and slightly generalized in~\cite{KP}, which apply only to the case when the non-wandering set is finite (so every non-wandering point is actually periodic). The condition regarding the finite number of non-wandering points comes naturally, because of the aforementioned growth of wandering orbits, which is polynomial.

In the present paper, we prove that the polynomial entropy of the dynamical system $(\mathcal{S}F_n(X),\mathcal{S}F_n(f))$ is equal to the polynomial entropy of $(F_n(X),F_n(f))$, and that they are both equal $n h_{pol}(f,X)$, under the assumptions that $X$ is compact and $f:X\to X$ is a homeomorphism with a finite chain recurrent set. Since it is needed that $F_n(f)$ has a finite non-wandering set in order to apply the aforementioned method, the assumptions for map $f$ have to be a little stronger than $NW(f)$ being finite. One sufficient condition is that $\mathcal{CR}(f)$ is finite. If $X$ is, for example, a regular curve, we have that actually $\mathcal{CR}(f)=NW(f)$ (see \cite{N1}). It is worth mentioning that dendrites and graphs are regular curves.

Comparing to the results obtained in~\cite{DK} for $(F_n(f),F_n(X))$, we generalise them for an arbitrary homeomorphism on a compact $X$ with a finite chain recurrent set and complement them with results regarding the suspension (and subsequently the generalised suspension).

\begin{theorem-non}
Let $X$ be a compact space, $f:X\to X$ a homeomorphism with a finite chain recurrent set and $n\geqslant2$. Then $$h_{pol}(\mathcal{S}F_n(f),\mathcal{S}F_n(X))=h_{pol}(F_n(f),F_n(X))=n h_{pol}(f,X).$$
\end{theorem-non}
Furthermore, in contrast to the condition of a finite chain recurrent set, we turn our attention to homeomorphisms with at least one wandering point. We obtain a similar result to that in \cite{DKL}:
\begin{theorem-non}
Let $X$ be a compact space and $f:X\to X$ be a homeomorphism with as least one wandering point $x_0$. If the sets $\alpha_{f}(x_0)$ and $\omega_{f}(x_0)$ are finite, then $h_{pol}(\mathcal{S}F_n(f))\geqslant n$, for all $n\geqslant2$.
\end{theorem-non}
Finally, we apply all the results to the setting when $X=[0,1]$ and $X=\mathbb{S}^1$.

\section{Preliminaries}

\subsection{Polynomial entropy and coding}\label{subsec:entropy}

Suppose that $X$ is a compact metric space, and $f:X\rightarrow X$ is continuous. Denote by $d_n^f(x,y)$ the dynamic metric (induced by $f$ and $d$):
$$
d_n^f(x,y)=\max\limits_{0\leqslant k\leqslant n-1}d(f^k(x),f^k(y)).
$$
Fix $Y\subseteq X$. For $\varepsilon>0$, we say that a finite set $E\subset X$ is \textit{$(n,\varepsilon)$-separated} if for every $x,y \in E$ it holds $d_n^f(x,y)\geqslant \varepsilon$. Let $S(n,\varepsilon;Y)$ denote the maximal cardinality of an $(n,\varepsilon)$-separated set $E$, contained in $Y$.

\begin{def}\label{def:pol_ent} The \textit{polynomial entropy} of the map $f$ on the set $Y$ is defined by
$$
h_{pol}(f;Y)=\lim\limits_{\varepsilon \rightarrow 0}\limsup\limits_{n\rightarrow \infty}\frac{\log S(n,\varepsilon;Y)}{\log n}.
$$
\end{def}
If $X=Y$, we abbreviate $h_{pol}(f):=h_{pol}(f;X)$.
The polynomial entropy, as well as the topological entropy, can also be defined via coverings with sets of $d^n_f$-diameters less than $\varepsilon$, or via coverings by balls of $d_n^f$-radius less than $\varepsilon$, see p.\ $626$ in $\cite{M2}$. We list some properties of the polynomial entropy that are important for our computations (for the proofs see Propositions $1-4$ in \cite{M2}).
\begin{itemize}
\item[(1)] $h_{pol}(f^k)=h_{pol}(f)$, for any $k\geqslant 1$.
\item[(2)] If $Y\subset X$ is a closed, $f$-invariant set, then $h_{pol}(f;Y)=h_{pol}(f|_Y)$.
\item[(3)] If $Y=\bigcup_{j=1}^mY_j$ where $Y_j$ are $f$-invariant, then $h_{pol}(f;Y)=\max\{h_{pol}(f;Y_j)\mid j=1,\ldots,m\}$.\label{finite-union}
\item[(4)] If $f:X\to X$, $g:Y\to Y$ and $f\times g:X\times Y\to X\times Y$ is defined as $(f\times g) (x,y):=(f(x),g(y))$, then $h_{pol}(f\times g)=h_{pol}(f)+h_{pol}(g)$.
\item[(5)] $h_{pol}(f)$ does not depend on a metric but only on the induced topology.
\item[(6)] $h_{pol}(\cdot)$ is a \textit{conjugacy invariant} (meaning if $f:X\to X$, $g:X'\to X'$, $\varphi:X\to X'$ is a homeomorphism of compact spaces and $g\circ\varphi=\varphi\circ f$, then $h_{pol}(f)=h_{pol}(g)$).
\item[(7)] If $f:X\to X$ and $g:X'\to X'$ are \textit{semi-conjugated}, meaning that $\varphi:X\to X'$ is a continuous surjective map of compact spaces and $g\circ\varphi=\varphi\circ f$, then $h_{pol}(f)\geqslant h_{pol}(g)$.
\end{itemize}

For $x\in X$, $\omega$-{\it limit set} is defined as:
$$\omega_f(x):=\left\{y\in X\mid \exists \;\mbox {a strictly increasing sequence}\;\mathbb{N}\ni n_k\to\infty\;\,f^{n_k}(x)\to y\right\},$$
and, for a reversible map $f$, $\alpha$-{\it limit set} is defined as:
$$\alpha_f(x):=\left\{y\in X\mid \exists \;\mbox {a strictly increasing sequence}\;\mathbb{N}\ni n_k\to\infty\;\,f^{-n_k}(x)\to y\right\}.$$
If $X$ is compact, the set $\omega_f(x)$ is nonempty, closed and $f$-invariant. The same is true for $\alpha(x)$, when $f$ is reversible.

A point $x\in X$ is \textit{non-wandering} if for every neighborhood $U$ of $x$ there is $n\geqslant1$ such that $f^n(U)\cap U$ is nonempty. We denote the set of all non-wandering points by $NW(f)$. The set $NW(f)$ is closed and $f$-invariant. Also, we denote the set of all fixed points by $Fix(f)$.

We now give a brief description of the coding procedure used for computing the polynomial entropy for maps with a finite non-wandering set. These results were first obtained in~\cite{HL} for homeomorphisms with only one non-wandering (hence fixed) point, but were later generalised in~\cite{KP} for continuous maps with finitely many non-wandering points.\\

Let $Y$ be any $f$-invariant subset of $X$. We first define a coding relative to a family of sets $\mathcal{F}$.  Let
$$\mathcal{F}=\{Y_1,Y_2,\ldots,Y_L\}$$ where $Y_j\subseteq X\setminus\mathrm{NW}(f)$ and
$$Y_{\infty}:=Y\setminus\bigcup_{j=1}^LY_j.$$ Let $\underline{x}=(x_0,\ldots,x_{n-1})$ be a finite sequence of elements in $Y$. We say that a finite sequence $\underline{w}=(w_0,\ldots,w_{n-1})$ of elements in $\mathcal{F}\cup\{Y_\infty\}$ is a {\it coding} of $\underline{x}$ relative to $\mathcal{F}$ if $x_j\in w_j$, for every $j=0,\ldots,n-1$. We will refer to $\underline{w}$ as a word and to $w_j$ as a letter.

Let $\mathcal{A}_n(\mathcal{F};Y)$ be the set of all codings of all orbits
$$(x,f(x),\ldots,f^{n-1}(x))$$ of length $n$ relative to $\mathcal{F}$, for all $x\in Y$. If
$\sharp \mathcal{A}_n(\mathcal{F};Y)$ denotes the cardinality of $\mathcal{A}_n(\mathcal{F};Y)$, we define the {\it polynomial entropy of $f$, on the set $Y$, relative to the family} $\mathcal{F}$ as the number:
$$h_{pol}(f,\mathcal{F};Y):=\limsup_{n\to\infty}\frac{\log \sharp \mathcal{A}_n(\mathcal{F};Y)}{\log n}.$$

We abbreviate $h_{pol}(\mathcal{F};Y):= h_{pol}(f,\mathcal{F};Y)$ whenever there is no risk of confusion. When we are dealing with a noncompact $Y$, we can relate the polynomial entropy $h_{pol}(f;Y)$ and the polynomial entropy relative to a one point family $\mathcal{F}=\{K\cap Y\}$, for all compact $K\subseteq X$:

\begin{prop}\label{prop1}{\rm [Proposition 3.2 in \cite{KP}]} Let $Y\subseteq X$ be an $f$-invariant set containing exactly one non-wandering point. Then it holds:
$$h_{pol}(f;Y)=\sup\{h_{pol}(\{K\cap Y\};Y)\mid{K\subseteq X\setminus\operatorname{NW}(f),\,K \,\mathrm{compact}}\}.$$\qed
\end{prop}

The polynomial entropy is a conjugacy invariant on compact spaces, but we can generalize that property in some instances for non-compact spaces, too. We used this generalization implicitly in the proof of the main Theorem A in \cite{DK}. The previous proposition will be used in order to prove the following result:

\begin{prop}\label{prop2} Let $f_1:X_1\to X_1$ and $f_2:X_2\to X_2$ be homeomorphisms on compact metric spaces $X_1$ and $X_2$. Let $A\subset X_1$ and $B\subset X_2$ be invariant subsets that contain exactly one non-wandering point of the map $f_1$ and $f_2$, respectively. If there exists a homeomorphism $H:A\to B$, then $h_{pol}(f_1;A)=h_{pol}(f_2;B)$.
\end{prop}
\begin{proof}
Note that the homeomorphism $H$ induces a bijection between $$\{K\cap A\mid K\subset X_1\setminus NW\left(f_1\right),\;K\;\text{compact}\}$$ and $$\{L\cap B\mid L\subset X_2\setminus NW\left(f_2\right),\;L\;\text{compact}\}.$$

If $\underline{w}=(w_0,\ldots,w_{n-1})$ is a coding of an orbit $\left(x,f_1(x),\ldots,(f_1)^{n-1}(x)\right)$ in
$A$ (consisting of letters $K\cap A$ and $Y_\infty:=A\setminus K$), then $(H(w_0),\ldots,H(w_{n-1}))$ is the coding of an orbit $\left(H(x),f_2(H(x)),\ldots,(f_2)^{n-1}(H(x))\right)$ in $B$ (consisting of letters $H(K\cap A)$ and $Y_\infty:=X_2\setminus H(K\cap A)$), and vice versa. Therefore, for a fixed compact $K$, the sets $\mathcal{A}_n\left(\{K\cap A\};A\right)$ and $\mathcal{A}_n\left(\{H(K\cap A)\};B\right)$ have the same cardinality. We finish the proof by applying Proposition \ref{prop1}.
\end{proof}

\subsection{Hyperspaces and induced maps}

For a compact metric space $(X,d)$, the hyperspace $2^X$ is the set of all nonempty closed subsets of $X$. The topology on $2^X$ (namely, the Vietoris topology) is induced by the Hausdorff metric
$$d_H(A,B):=\inf\{\varepsilon>0\mid A\subset U_\varepsilon(B),\;B\subset U_\varepsilon(A)\},$$
where
\begin{equation}\label{eq:neighb}
U_\varepsilon(A):=\{x\in X\mid d(x,A)<\varepsilon\}.
\end{equation}
We consider a closed subspace $F_n(X)$ of $2^X$, $n\geqslant1$, consisting of all finite nonempty subsets of cardinality at most $n$, with the induced metric. The set $F_n(X)$ is called the \textit{$n$-fold symmetric product of $X$}. Both spaces $2^X$ and $F_n(X)$ are compact metric spaces with respect to the Hausdorff metric.

If $f:X\to X$ is continuous, then it induces continuous maps
 $2^f:2^X\to 2^X$ and $F_n(f):F_n(X)\to F_n(X)$, $F_n(f)=2^f\mid_{F_n(X)}$ by:
 $$2^f(A):=\{f(x)\mid x\in A\}.$$
 If $f$ is a homeomorphism, then this is also true for both $2^f$ and $F_n(f)$. Because of the following commutative diagram, it is clear that $(F_n(f),F_n(X))$ is a factor of $(f^{\times n},X^{\times n})$, so $h_{pol}(F_n(f),F_n(X))\leqslant h_{pol}(f^{\times n},X^{\times n})=n h_{pol}(f,X)$.
 \[ \begin{tikzcd}
X^{\times n} \arrow{r}{f^{\times n}} \arrow[swap]{d}{\pi} & X^{\times n} \arrow{d}{\pi} \\%
F_n(X)\arrow{r}{F_n(f)}& F_n(X)
\end{tikzcd}
\]
Here, $X^{\times n}=\underbrace{X\times\dots\times X}_{n}$, $f^{\times n}(x_1,\ldots,x_n)=(f(x_1),\ldots,f(x_n))$ and\\ $\pi(x_1,\ldots,x_n)=\{x_1,\ldots,x_n\},$ which is a continuous surjective map.
\begin{exam}
If $X=[0,1]$, then $F_2(X)$ can be identified with the triangle whose vertices are $(0,0),(0,1),(1,1)$ and $F_1(X)$ is homeomorphic to the line segment joining $(0,0)$ and $(1,1)$. If $X=\mathbb{S}^1$, then $F_2(X)$ is homeomorphic to the M$\ddot{o}$bius band and $F_1(X)$ is homeomorphic to the manifold boundary of the M$\ddot{o}$bius band.
\end{exam}

\subsection{Suspensions} Let $X$ be a compact metric space and $n\in\mathbb{N}$, $n\geqslant2$.
\begin{defn}
The $n$-fold symmetric product suspension of the continuum $X$ is the quotient space $$\mathcal{S}F_n(X)=F_n(X)/F_1(X),$$ with the quotient topology.
\end{defn}
We denote the quotient map by $q:F_n(X)\to\mathcal{S}F_n(X)$ and $q(F_1(X))=F_{X}$. Let us remark that $$\mathcal{S}F_n(X)=\{\{A\}\mid A\in F_n(X)\setminus F_1(X)\}\cup\{F_X\},$$ and that using the appropriate restriction of $q$, we have that $\mathcal{S}F_n(X)\setminus\{F_X\}$ is homeomorphic to $F_n(X)\setminus F_1(X)$. It is a well known fact that $\mathcal{S}F_n(X)$ is also a compact metric space (see $3.10$ in \cite{SN}). 

Let us now define the induced map of $f$ on the $n$-fold symmetric product suspension $\mathcal{S}F_n(f):\mathcal{S}F_n(X)\to\mathcal{S}F_n(X)$, by:
\begin{equation}
    \mathcal{S}F_n(f)(\mathcal{A}) = 
    \left\{
        \begin{array}{ll}
            q(F_n(f)(q^{-1}(\mathcal{A}))), & \text{if } \mathcal{A} \neq F_{X}; \\
            F_X, & \text{if } \mathcal{A}=F_X.
        \end{array}
        \right. 
\end{equation}

If $f:X\to X$ is continuous, then $\mathcal{S}F_n(f)$ is also a continuous map (see p. 126 in \cite{JD}) and the following diagram commutes.
\[ \begin{tikzcd}
F_n(X) \arrow{r}{F_n(f)} \arrow[swap]{d}{q} & F_n(X) \arrow{d}{q} \\%
\mathcal{S}F_n(X)\arrow{r}{\mathcal{S}F_n(f)}& \mathcal{S}F_n(X)
\end{tikzcd}
\]

If $f:X\to X$ is a homeomorphism, then so is $\mathcal{S}F_n(f)$. Namely, since $q$ is injective on $F_n(X)\setminus F_1(X)$, and $\mathcal{S}F_n(f)^{-1}(F_X)=F_1$, we conclude that $\mathcal{S}F_n(f)$ is injective. Since $q$ is surjective, $\mathcal{S}F_n(f)$ is a continuous bijective map from a compact space to a Hausdorff space, which is known to be a homeomorphism. At this point we already see that $(\mathcal{S}F_n(f),\mathcal{S}F_n(X))$ is a factor of $(F_n(f),F_n(X))$, so $h_{pol}(\mathcal{S}F_n(f))\leqslant h_{pol}(F_n(f))$. Our goal is to prove $h_{pol}(\mathcal{S}F_n(f))\geqslant h_{pol}(F_n(f))$.

\begin{exam}
If $X=[0,1]$, then $SF_2(X)$ is obtained by collapsing  $F_1(X)$ to a point, and this space is again homeomorphic to a triangle. If $X=\mathbb{S}^1$, then $SF_2(X)$ is homeomorphic to the quotient space of M$\ddot{o}$bius band, with its boundary collapsed, which is known as a model for the real projective space $\mathbb{R}P^2$.
\end{exam}

\section{Finiteness of $NW(F_n(f))$}

As we already mentioned, in order to apply the method described in~\cite{HL} and in~\cite{KP}, the non-wandering set of $F_n(f)$ has to be finite. As we can see from Examples 14 and 15 in \cite{RIM}, the non-wandering set of the induced map $F_n(f)$ is not directly connected with the non-wandering set of $f$, as one might think. One way to achieve finiteness of $NW(F_n(f))$ is that we assume that $\mathcal{CR}(f)$ is finite, as we will now prove. Let us start with an auxiliary lemma.

\begin{lem}\label{lema}
Let $X$ be a compact metric space and $f:X\to X$ a continuous map. Let $A,B\in F_N(X)$, $A=\{a_1,\ldots,a_k\}$ and $B=\{b_1,\ldots,b_m\}$, $1\leqslant m,k\leqslant N$. There exists $\varepsilon_0 $ such that for every $\varepsilon<\varepsilon_0$ the following is true: if $d_H(A,B)<\varepsilon$, then there exist $N-$tuples $(\alpha_1,\ldots,\alpha_N)\in\pi^{-1}(A)$ and $(\beta_1,\ldots,\beta_N)\in\pi^{-1}(B)$ such that $d((\alpha_1,\ldots,\alpha_N),(\beta_1,\ldots,\beta_N))<\varepsilon$.
\end{lem}

\begin{proof}

Let $\delta:=\min\limits_{i\neq j}{d(a_i,a_j)}$, $\varepsilon_0:=\delta/2$, $\varepsilon<\varepsilon_0$ and suppose that $d_H(A,B)<\varepsilon$. Then, from the definition of Hausdorff metric, we have that
$$(\forall i\in\{1,\ldots,k\})(\exists j_i\in\{1,\ldots,m\})\ d(a_i,b_{j_i})<\varepsilon.$$

Note that for $i\neq l$, we have that $b_{j_i}\neq b_{j_l}$. Indeed,

\begin{align*}
\delta\leqslant d(a_i,a_l)&\leqslant d(a_i,b_{j_i})+d(b_{j_i,b_{j_l}})+d(b_{j_l},a_l)\\
&<2\varepsilon+d(b_{j_i},b_{j_l})<\delta+d(b_{j_i},b_{j_l}).
\end{align*}

Therefore $d(b_{j_i},b_{j_l})>0$, and consequently $b_{j_i}\neq b_{j_l}$. We conclude that it has to be $m\geqslant k$.\\

We can now explain how we define $(\alpha_1,\ldots,\alpha_N)\in\pi^{-1}(A)$ and $(\beta_1,\ldots,\beta_N)\in\pi^{-1}(B)$. Firstly, $(\alpha_1,\ldots,\alpha_k)=(a_1,\ldots,a_k)$ and $(\beta_1,\ldots,\beta_k)=(b_{j_1},\ldots,b_{j_k})$. Then let $\{\beta_{j_k+1},\ldots,\beta_{m}\}$ be the elements from $\{b_1,\ldots,b_m\}\setminus\{b_{j_1},\ldots,b_{j_k}\}$. The corresponding $\{\alpha_{j_k+1},\ldots,\alpha_{m}\}$ can be chosen from $\{a_1,\ldots,a_k\}$, in such way that $d(\alpha_p,\beta_p)<\varepsilon$, for $p\in\{j_k+1,\ldots,m\}$. Indeed, if this weren't the case, then $d_H(A,B)\geqslant\varepsilon$. For the remainder of the pairs of coordinates, if there are any ($m<N$), we can just repeat any of the already used pairs. In this way we obtained $N$-touples which are close, when their projections in $F_N(X)$ are close.

\end{proof}

\begin{thrm}\label{CR}
Let $X$ be a compact metric space and $f:X\to X$ a continuous map such that $\mathcal{CR}(f)=\mathrm{Fix}(f)=\mathrm{NW}(f)$ is a finite set. Then 
$$\mathrm{NW}(F_N)(f)=\mathrm{Fix}(F_N)=\big\{\{x_1,\ldots,x_k\}\mid 1\le k\le N,\,x_j\in\mathrm{Fix}(f)\big\}.$$
\end{thrm}

\begin{proof}

Let $A\in\mathrm{NW}(F_N)(f)$, $A=\{a_1,\ldots,a_k\}$, $k\leqslant N$. If we show that $(a_1,\ldots,a_k)\in\mathcal{CR}(f^{\times k})$, $f^{\times k}:X^{\times k}\to X^{\times k}$, then using the fact that $\mathcal{CR}(f\times g)=\mathcal{CR}(f)\times\mathcal{CR}(g)$, we have that $a_i\in\mathcal{CR}(f)$, $1\leqslant i\leqslant k$. Since $\mathcal{CR}(f)=\mathrm{Fix}(f)$, we have that $A\in\mathrm{Fix}(F_N(f))$. Let us show that $(a_1,\ldots,a_k)\in\mathcal{CR}(f^{\times k})$. More precisely, we will construct an $\varepsilon-$chain starting and ending at $(a_1,\ldots,a_k)$, for an arbitrary $\varepsilon>0$. Note that it is enough to find an $\varepsilon-$chain starting and ending at any permutation of $(a_1,\ldots,a_k)$.\\

Since is $A$ is non-wandering, for every $n\in\mathbb{N}$, choose $U_n:=B(A,1/n)$. There exists $B_n\in F_N(X)$ and $m_n\in\mathbb{N}$ such that
$$d(A,B_n)<1/n,\quad d(A,F_N^{m_n}(f)(B_n))<1/n.$$

We can also choose $(m_n)_{n\in\mathbb{N}}$ to be strictly increasing sequence.

Denote by $\pi^{-1}(A)=\{\mathbf{a}_1,\ldots,\mathbf{a}_p\}\in X^N$, $p\leqslant P$ and by $\pi^{-1}(B_n)=\{\mathbf{b}_1^n,\ldots,\mathbf{b}_{q(n)}^n\}\in X^N$, $q(n)\leqslant Q$. Using Lemma~\ref{lema}, we can conclude that for every $n\in\mathbb{N}$, there exist $i_n,j_n\in\{1,\ldots,p\}$ and $l_n,r_n\in\{1,\ldots,q(n)\}$ such that
 $$d(\mathbf{a}_{i_n},\mathbf{b}_{r_n}^{n})<1/n, \quad d\left(\mathbf{a}_{j_n},{(f^{\times N})}^{m_n}(\mathbf{b}_{l_n}^{n})\right)<1/n.$$
 
Since the set $\{1,\ldots,p\}$ is finite, there exists $i_0\in\{1,\ldots,p\}$ and a subsequence $p_n$ of $\mathbb{N}$ such that
 
 $$d(\mathbf{a}_{i_0},\mathbf{b}_{r_{p_n}}^{p_n})<1/p_n.$$ 

Similarly, there exists $r_0\in\{1,\ldots,Q\}$ and a subsequence $(k_n)$ of $(p_n)$ such that 

 $$d(\mathbf{a}_{i_0},\mathbf{b}_{r_{0}}^{k_n})<1/k_n.$$ 

Now we have  

$$d\left(\mathbf{a}_{j_{k_n}},{(f^{\times N})}^{m_{k_n}}(\mathbf{b}_{l_{k_n}}^{k_n})\right)<1/k_n.$$

By repeating the same reasoning once again, we find a new subsequence of $(k_n)$, denote it with $(s_n)$, and $j_0\in\{1,\ldots,m\}$ such that
$$d\left(\mathbf{a}_{j_0},{(f^{\times N})}^{m_{s_n}}(\mathbf{b}_{l_{s_n}}^{s_n})\right)<1/s_n.$$
 
Furthermore, as before, there exists $l_0\in\{1,\ldots,Q\}$ and a subsequence $(t_n)$ of $(s_n)$ such that

$$d\left(\mathbf{a}_{j_0},{(f^{\times N})}^{m_{t_n}}(\mathbf{b}_{l_0}^{t_n})\right)<1/t_n.$$

To summarize, we have

 $$d(\mathbf{a}_{i_0},\mathbf{b}_{r_{0}}^{t_n})<1/t_n, \quad d\left(\mathbf{a}_{j_0},{(f^{\times N})}^{m_{t_n}}(\mathbf{b}_{l_0}^{t_n})\right)<1/t_n.$$ 
 
If we denote $\mathbf{a}_{i_0}=(\alpha_1,\ldots,\alpha_N)$ and $\mathbf{b}_{r_{0}}^{t_n}=(\beta^{t_n}_1,\ldots,\beta^{t_n}_N)$, then $d(\alpha_j,\beta^{t_n}_j)<1/t_n$, for all $1\leqslant j\leqslant N$. Similarly, if we denote  $\mathbf{a}_{j_0}=(\gamma_1,\ldots,\gamma_N)$ and $\mathbf{b}_{l_0}^{t_n}=(\delta^{t_n}_1,\ldots,\delta^{t_n}_N)$, we have that $d(\gamma_j,f^{m_{t_n}}(\delta^{t_n}_j))<1/t_n$, for all $1\leqslant j\leqslant N$.\\
 
 $N-$tuples $\mathbf{a}_{i_0}$ and $\mathbf{a}_{j_0}$ both have exactly $k$ mutually different coordinates. We discard all the sequences $(\beta^{t_n}_j)$, and alike $f^{m_{t_n}}(\delta^{t_n}_j)$, that have the same limit, and leave just one with the said limit. We are left with $k$ sequences $(\beta^{t_n}_j)$ and $k$ sequences $(f^{m_{t_n}}(\delta^{t_n}_j))$. We will now use simpler notation for the obtained inequalities. We have $d(p_j,c^n_j)<1/t_n$ and $d(r_j,f^{u_n}(d^n_j))<1/t_n$, for all $1\leqslant j\leqslant k$. Note that $A=\{a_1,\ldots,a_k\}=\{p_1,\ldots,p_k\}=\{r_1,\ldots,r_k\}$ and $\{c^n_1,\ldots,c^n_k\}=\{d^n_1,\ldots,d^n_k\}$. As before, denote $\mathbf{p}=(p_1,\ldots,p_k)$, $\mathbf{r}=(r_1,\ldots,r_k)$, $\mathbf{c}^n=(c^n_1,\ldots,c^n_k)$ and $\mathbf{d}^n=(d^n_1,\ldots,d^n_k)$. \\

Let $\varepsilon>0$ and choose $\delta<\varepsilon/2$ such that 
\begin{equation}\label{1}
 d(x,y)<\delta\Rightarrow d(f^{\times k}(x),f^{\times k}(y))<\varepsilon
\end{equation}
and $n_0\in\mathbb{N}$ such that 
$$d(\mathbf{p},\mathbf{c}^n)<\delta, \quad d(\mathbf{r},(f^{\times k})^{u_n}(\mathbf{d}^n))<\varepsilon/2$$
for $n\geqslant n_0$. This also implies that for every $s\in\{0,\ldots,l-1\}$ and $n\geqslant n_0$
\begin{equation}\label{5}
d(\mu^s(\mathbf{r}),(f^{\times k})^{u_n}(\mu^s(\mathbf{d}^n)))<\varepsilon/2, 
\end{equation}
since the distance doesn't change if we permute both $k-$tuples in the same way.

 Let $\nu$ be the permutation in $S_k$ mapping $(c^{n_0}_1,\ldots,c^{n_0}_k)=\mathbf{c}^{n_0}$ to $(d^{n_0}_1,\ldots,d^{n_0}_k)=\mathbf{d}^{n_0}$ and $m$ such that $\nu^m=Id$. Denote $\mathbf{p}_1=\nu(\mathbf{p})$.  Let $\mu$ be the permutation in $S_k$ mapping $\mathbf{p}_1$ to $(r_1,\ldots,r_k)=\mathbf{r}$ and $l$ such that $\mu^l=Id$.\\
 
 We have that 
 \begin{equation}\label{2}
 d(\mathbf{p},\mathbf{c}^{n_0})<\delta\Leftrightarrow d(\nu(\mathbf{p}),\nu(\mathbf{c}^{n_0}))<\delta\Leftrightarrow d(\mathbf{p}_1,\mathbf{d}^{n_0})<\delta.
 \end{equation}

Applying $(\ref{1})$ to $(\ref{2})$, we obtain 
\begin{equation}\label{3}
d(f^{\times k}(\mathbf{p}_1),f^{\times k}(\mathbf{d}^{n_0}))<\varepsilon.
\end{equation}

Furthermore,
\begin{align*}
d(\mu(\mathbf{d}^{n_0}),(f^{\times k})^{u_{n_0}}(\mathbf{d}^{n_0}))
&\leqslant d(\mu(\mathbf{d}^{n_0}),\mathbf{r})+d(\mathbf{r},(f^{\times k})^{u_{n_0}}(\mathbf{d}^{n_0}))\\
&=d(\mathbf{d}^{n_0},\mu^{-1}\mathbf{r})+d(\mathbf{r},(f^{\times k})^{u_{n_0}}(\mathbf{d}^{n_0}))\\
&=d(\mathbf{d}^{n_0},\mathbf{p}_1)+d(\mathbf{r},(f^{\times k})^{u_{n_0}}(\mathbf{d}^{n_0}))\\
&<\varepsilon/2+\varepsilon/2=\varepsilon,
\end{align*}
Therefore, for every $s\in\{0,\ldots,l-1\}$, we have that 
\begin{equation}\label{4}
d(\mu^s(\mathbf{d}^{n_0}),(f^{\times k})^{u_{n_0}}(\mu^{s-1}(\mathbf{d}^{n_0})))<\varepsilon.
\end{equation}

Let us finally construct an $\varepsilon$-chain starting and ending at $\mathbf{p}_1$. We will explain why every first term in each line is preceded by the last term in the previous line in an $\varepsilon$-chain. Elements in each line are trivially elements of an $\varepsilon$-chain. We can go from the first to the second line because of $(\ref{3})$. For the following $l-2$ lines we apply $(\ref{4})$. And, lastly, we use $(\ref{5})$, for $s=l-1$.

$$\mathbf{p}_1,$$
$$f^{\times k}(\mathbf{d}^{n_0}), (f^{\times k})^{2}(\mathbf{d}^{n_0}),\ldots, (f^{\times k})^{u_{n_0}-1}(\mathbf{d}^{n_0}),$$
$$\mu(\mathbf{d}^{n_0}),f^{\times k}(\mu(\mathbf{d}^{n_0})),(f^{\times k})^{2}(\mu(\mathbf{d}^{n_0})),\ldots,(f^{\times k})^{u_{n_0}-1}(\mu(\mathbf{d}^{n_0})),$$
$$\mu^{2}(\mathbf{d}^{n_0}),f^{\times k}(\mu^2(\mathbf{d}^{n_0})),(f^{\times k})^{2}(\mu^2(\mathbf{d}^{n_0})),\ldots,(f^{\times k})^{u_{n_0}-1}(\mu^2(\mathbf{d}^{n_0})),$$
$$\ldots$$
$$\mu^{l-1}(\mathbf{d}^{n_0}),f^{\times k}(\mu^{l-1}(\mathbf{d}^{n_0})),(f^{\times k})^2(\mu^{l-1}(\mathbf{d}^{n_0})),\ldots,(f^{\times k})^{u_{n_0}-1}(\mu^{l-1}(\mathbf{d}^{n_0})),$$
$$\mu^{l-1}(\mathbf{r})=\mathbf{p}_1$$

\end{proof}

\section{Main results and proofs}

Let $f:X\to X$ be a homeomorphism with finitely many chain recurrent points. Since $\mathcal{CR}(f)$ is finite and invariant, we conclude it consists of only periodic points. We can pass from $f$ to $f^k$, for some $k\geqslant1$, so that all the chain recurrent points become fixed. This change does not affect the polynomial entropy, since we have that $h_{pol}(f^k)=h_{pol}(f)$. It is clear that this is also true for the map $F_n(f)$, i.e. $$h_{pol}(F_n(f^k))=h_{pol}(F_n(f)^k)=h_{pol}(F_n(f)).$$

Let $\mathcal{A}\in\mathcal{S}F_n(X)$. Then, since the appropriate restriction of $q$ is a homeomorphism, we have:

\begin{equation}
\begin{aligned}
    (\mathcal{S}F_n(f))^k(\mathcal{A})&=(qF_n(f)q^{-1})^k(\mathcal{A})\\
    &=
    \left\{
        \begin{array}{ll}
            qF_n(f)^kq^{-1}(\mathcal{A}) & \text{if } \mathcal{A} \neq F_{X} \\
            F_X & \text{if } \mathcal{A}=F_X
        \end{array}
        \right.\\
    &=\mathcal{S}F_n(f^k)(\mathcal{A}).
\end{aligned}
\end{equation}

We see that 
\begin{equation*}
\begin{aligned}
h_{pol}((\mathcal{S}F_n(f^k)),\mathcal{S}F_n(X))&=h_{pol}((\mathcal{S}F_n(f)^k),\mathcal{S}F_n(X))\\
&=h_{pol}(\mathcal{S}F_n(f),\mathcal{S}F_n(X)),
\end{aligned}
\end{equation*} 
so we can, from now on, consider homeomorphisms $f$ with $\mathcal{CR}(f)=Fix(f)$.

\begin{lem}
Let $X$ be a compact space and $f:X\to X$ a homeomorphism with a finite chain recurrent set $\mathcal{CR}(f)=Fix(f)$. Then $\mathcal{S}F_n(f))$ has a finite non-wandering set and $NW(\mathcal{S}F_n(f)))=Fix(\mathcal{S}F_n(f)))$, for all $n\geqslant2$.
\end{lem}

\begin{proof}
The set $Fix(\mathcal{S}F_n(f)))$ is finite because it is contained in a finite set $(Fix(F_n(X))\setminus\{F_1\})\cup \{F_X\}$. Let $\mathcal{A}\in\mathcal{S}F_n(X)$ such that $\mathcal{A}\in NW(\mathcal{S}F_n(f))\setminus Fix(\mathcal{S}F_n(f))$. Since $\mathcal{A}\neq F_X$, we have that $\mathcal{A}=q(A)$, for some $A$ with more than one element. Let $U$ be a neighbourhood of $A$ in $F_n(X)$. Then $q(U)$ is a neighbourhood of $\mathcal{A}$. Define $\mathcal{U}:=q(U)\setminus{F_X}$. Following that $\mathcal{A}$ in a non-wandering point, there exists $k\geqslant1$ so that $$(\mathcal{S}F_n(f))^k(\mathcal{U})\cap \mathcal{U}\neq\emptyset.$$

Then, following the definition of the map $\mathcal{S}F_n(f)$ and a simple property that $(\mathcal{S}F_n(f))^k\circ q=q\circ(F_n(f))^k$:
$$\emptyset\neq(\mathcal{S}F_n(f))^k(\mathcal{U})\cap \mathcal{U}=(\mathcal{S}F_n(f))^k(q(U))\cap q(U)=q((F_n(f))^k)(U)\cap q(U).$$

Since $q$ is a bijection on $F_n(X)\setminus F_1(X)$, we can conclude that:
$$(F_n(f))^k(U)\cap U\neq\emptyset.$$

Since $U$ was an arbitrary neighbourhood of $A$, we see that $A\in NW(F_n(X))$. From Theorem~\ref{CR}, we have that $NW(F_n(X))=Fix(F_n(X))$, and finally:
$$\mathcal{S}F_n(f)(\mathcal{A})=q(F_n(f))q^{-1}(\mathcal{A})=q(F_n(f))(A)=q(A)=\mathcal{A},$$
which shows that $\mathcal{A}\in Fix(\mathcal{S}F_n(f))$.
\end{proof}

\begin{thrm}\label{glavna}
Let $X$ be a compact space, $f:X\to X$ a homeomorphism with a finite chain recurrent set and $n\geqslant2$. Then $$h_{pol}(\mathcal{S}F_n(f),\mathcal{S}F_n(X))=h_{pol}(F_n(f),F_n(X))=n h_{pol}(f,X).$$
\end{thrm}

\begin{proof}
As we explained, we can consider all the chain recurrent (and then consequently non-wandering) points to be fixed points. Let $$G_n:=\{\{x_1,\ldots,x_n\}\mid x_i\in X,x_i\neq x_j,i,j\in\{1,\ldots n\}\}\subset F_n(X).$$ Since $f$ is a bijection, $G_n$ is $f$-invariant. Similarly, $\mathcal{S}F_n(f)$ and the restriction of $q$ are bijections, so $q(G_n)$ is $\mathcal{S}F_n(f)$-invariant. 
We are going to apply Proposition \ref{prop2} for: 

\begin{align*}
X_1:&=F_n(X), & A:&=G_n, & f_1:&=F_n(f),\\
X_2:&=\mathcal{S}F_n(X), & B:&=q(G_n), & f_2:&=\mathcal{S}F_n(f),\\
H:&=q. & & & &
\end{align*}
  
If $f$ has exactly $n$ fixed points $\{x_1,\ldots,x_n\}$, then the sets $G_n$ and $q(G_n)$ contain exactly one non-wandering point $\{x_1,\ldots,x_n\}$. If $f$ has less than $n$ fixed points, then the sets $G_n$ and $q(G_n)$ do not contain any non-wandering points. In this case, we will add a one-element set $\{\tilde{x}\}$, $f(\tilde{x})=\tilde{x}$ to $G_n$ and $F_1$ to $q(G_n)$, while keeping the same notation, to meet the requirements of Proposition \ref{prop2}. This will not affect the polynomial entropies (see property $(3)$ on page \pageref{finite-union}) and $q$ will remain a homeomorphism on $G_n$. Also, $G_n$ and $q(G_n)$ remain $f$- and $\mathcal{S}F_n(f)$-invariant, respectively. If $f$ has more than $n$ fixed points, say $k$, then there are $\binom{k}{n}$ non-wandering points in $G_n$. Again, without changing the notation, we remove all the non-wandering points from $G_n$, and add  a one-element set $\{\tilde{x}\}$, $f(\tilde{x})=\tilde{x}$ to $G_n$ and $F_1$ to $q(G_n)$. Since every removed point is fixed, the remaining points in $G_n$ constitute an $f$-invariant set. The same is true for $q(G_n)$, being $\mathcal{S}F_n(f)$-invariant.

Seeing that all the conditions are fulfilled, we have:
  
\begin{equation}\label{eq:1}
h_{pol}(\mathcal{S}F_n(f),q(G_n))=h_{pol}(F_n(f),G_n).
\end{equation}
  
Once again, we will apply Proposition \ref{prop2}, but this time for the following setting:
\begin{align*}
X_1:&=X^{\times n},& A:&=\pi^{-1}(G_n),& f_1:&=f^{\times n},\\
X_2:&=F_n(X), & B:&=G_n,&f_2:&=F_n(f),\\
H:&=\pi.& & & &
\end{align*}
  
First note that $\pi^{-1}(G_n)$ is $f^{\times n}$-invariant, since $f$ and the restriction of $\pi$ are bijections. Following the same reasoning as before, there is either exactly one non-wandering point, or there is none (either instantaneously, or after removing all the fixed points), so we add one point to the sets $\pi^{-1}(G_n)$ and $G_n$ - point $(\tilde{x},\ldots,\tilde{x})$ and $\{\tilde{x}\}$, respectively. $H$ will remain a homeomorphism, and the entropies are not changed. We can conclude that
  
\begin{equation}\label{eq:2}
h_{pol}(F_n(f),G_n)=h_{pol}(f^{\times n},\pi^{-1}(G_n)).
\end{equation}

Using the elementary property of polynomial entropy, as well as (\ref{eq:1}) and (\ref{eq:2}), we have the following:
  
\begin{equation*}
\begin{aligned}
h_{pol}(\mathcal{S}F_n(f),\mathcal{S}F_n(X))&\geqslant h_{pol}(\mathcal{S}F_n(f),q(G_n))\\
&=h_{pol}(F_n(f),G_n)\\
&=h_{pol}(f^{\times n},\pi^{-1}(G_n)).
\end{aligned}
\end{equation*}

Therefore, it is enough to show that 
  
\begin{equation}\label{eq:3}
h_{pol}(f^{\times n},\pi^{-1}(G_n))=h_{pol}(F_n(f),F_n(X)).
\end{equation}
  
Note that if we prove that $h_{pol}(f^{\times n},\pi^{-1}(G_n))=n h_{pol}(f,X)$, then (\ref{eq:3}) is true. Namely,
  
\begin{equation}\label{eq:4}
\begin{aligned}
  h_{pol}(F_n(f),F_n(X))&\leqslant h_{pol}(f^{\times n},X^{\times n})\\
  &=n h_{pol}(f,X),
\end{aligned}
\end{equation}

\begin{equation}\label{eq:5}
\begin{aligned}
  h_{pol}(F_n(f),F_n(X))&\geqslant h_{pol}(F_n(f),G_n)\\
  &\stackrel{(\ref{eq:2})}=h_{pol}(f^{\times n},\pi^{-1}(G_n))\\
  &=n h_{pol}(f,X),
\end{aligned}
\end{equation}

so we conclude that $h_{pol}(F_n(f),F_n(X))=n h_{pol}(f,X)$.\\ 

Let us now present the proof of the following equality:

\begin{lem}
$h_{pol}(f^{\times n},\pi^{-1}(G_n))=n h_{pol}(f,X)$.
\end{lem}

\begin{proof}
One inequality is trivially true:
$$h_{pol}(f^{\times n},\pi^{-1}(G_n))\leqslant h_{pol}(f^{\times n},X^{\times n})=n h_{pol}(f,X).$$

Notice that the set $\pi^{-1}(G_n)$ is the set of $n$-tuples where all the coordinates are mutually different. Hence, 
$$ X^{\times n}=\pi^{-1}(G_n)\cup(\cup G_{k,\alpha(k)})\cup(Fix(f^{\times n})\cap\pi^{-1}(G_n)),$$ where the union is taken over all $k$ and $\alpha(k)$,
with $k\geqslant2$ being the least number of mutually equal coordinates and $\alpha(k)$ the set of parameters which determines the positions of these coordinates. For clarity, let us explicitly write one set $G$ from the family of sets $G_{k,\alpha(k)}$:
$$G:=G_{2,(1,2)}\{(x,x,x_3,\ldots,,x_n)\mid x,x_i\in X, i\in\{3,\ldots,n\}\}.$$
There are finitely many sets $G_{k,\alpha(k)}$ for a fixed $n$ and they are all $f^{\times n}$-invariant and compact. Let us remark that they are not necessarily disjoint, which is not a problem for using the finite union property of the polynomial entropy \ref{finite-union}. Since $h_{pol}(f^{\times n},Fix(f^{\times n})\cap\pi^{-1}(G_n))=0$, we have that:

$$n h_{pol}(f,X)=h_{pol}(f^{\times n},X^{\times n})=\max_{k,\alpha(k)}\{h_{pol}(f^{\times n},\pi^{-1}(G_n)),h_{pol}(f^{\times n},G_{k,\alpha(k)})\}.$$

We can conclude that $h_{pol}(f^{\times n},\pi^{-1}(G_n))=n h_{pol}(f,X)$ after we prove that $h_{pol}(f^{\times n},G_{k,\alpha(k)})\leqslant(n-1)h_{pol}(f,X)$.\\ Firstly, let us show that $h_{pol}(f^{\times n},G)=(n-1)h_{pol}(f,X)$. The following diagram, where $\phi(x_1,\ldots,x_{n-1})=(x_1,x_1,x_2,\ldots,x_{n-1})$ is a homeomorphism, establishes that $h_{pol}(f^{\times n},G)=h_{pol}(f^{\times (n-1)},X^{\times(n-1)})=(n-1)h_{pol}(f,X)$.
\[ \begin{tikzcd}
X^{\times(n-1)}\arrow{r}{f^{\times (n-1)}} \arrow[swap]{d}{\phi} & X^{\times(n-1)} \arrow{d}{\phi} \\
G\arrow{r}{f^{\times (n-1)}}& G
\end{tikzcd}
\]
Also, $h_{pol}(f^{\times n},G_{2,\alpha(k)})=h_{pol}(f^{\times n},G)$, for all $\alpha(k)$, because sets $G$ and $G_{2,\alpha(k)}$ are homeomorphic via the map that exchanges proper coordinates, so $h_{pol}(f^{\times n},G_{2,\alpha(k)})=(n-1)h_{pol}(f,X)$.\\
Finally, let $k>2$. Then $h_{pol}(f^{\times n},G_{k,\alpha(k)})=(n-k+1)h_{pol}(f,X)$, using the same reasoning as before, so we have that  $h_{pol}(f^{\times n},G_{k,\alpha(k)})\leqslant(n-1)h_{pol}(f,X)$.

\end{proof}  
This completes the proof of our main theorem, since we proved that $h_{pol}(\mathcal{S}F_n(f),\mathcal{S}F_n(X))\geqslant h_{pol}(F_n(f),F_n(X))$ as well as $h_{pol}(F_n(f),F_n(X))=n h_{pol}(f,X)$.
\end{proof}

Barragán et al. defined the space $(\mathcal{S}F^{m}_n(f),\mathcal{S}F^{m}_n(X))$, $n>m\geqslant1$ in \cite{BST2} as the quotient space $$\mathcal{S}F^{m}_n(X)=F_n(X)/F_m(X),$$ with the quotient topology. They proved that $\mathcal{S}F^{m}_n(f)$ is a continuous map on the compact $\mathcal{S}F^{m}_n(X)$, when $f$ is a continuous map on a compact space $X$. Copying the proof of the main Theorem \ref{glavna} verbatim, we can show that the same results are true for the dynamical system $(\mathcal{S}F^{m}_n(f),\mathcal{S}F^{m}_n(X))$.

\begin{thrm}
Let $X$ be a compact space, $f:X\to X$ a homeomorphism with a finite chain recurrent set and $n>m\geqslant1$. Then $$h_{pol}(\mathcal{S}F^{m}_n(f),\mathcal{S}F^{m}_n(X))=h_{pol}(F_n(f),F_n(X))=n h_{pol}(f,X).$$\qed
\end{thrm}

As we mentioned in the introduction, we obtained the results regarding the polynomial entropy on $F_n(X)$, for specific spaces $X=[0,1]$ and $\mathbb{S}^1$ in \cite{DK}. We complement them with the results on the symmetric product suspension, using the main Theorem \ref{glavna} and the fact that $h_{pol}(f,[0,1])=h_{pol}(f,\mathbb{S}^1)=1$ when the set $NW(f)$ is finite. Again, since $[0,1]$ and $\mathbb{S}^1$ are regular curves, in this case the non-wandering set and the chain recurrent set coincide.

\begin{cor}
Let $f:[0,1]\to [0,1]$ be a homeomorphism with finitely many non-wandering points. Then $h_{pol}(F_n(f))=h_{pol}(\mathcal{S}F_n(f))=n$, for all $n\geqslant2$.
\end{cor}

\begin{cor}
Let $f:\mathbb{S}^1\to \mathbb{S}^1$ be a homeomorphism with finitely many non-wandering points. Then $h_{pol}(F_n(f))=h_{pol}(\mathcal{S}F_n(f))=n$, for all $n\geqslant2$.
\end{cor}

Let us now consider a homeomorphism $f:X\to X$ on a compact space $X$ with a non-wandering set that is not necessarily finite. We studied homeomorphisms with at least one wandering point in \cite{DKL} and obtained a lower bound for the polynomial entropy of $(F_n(f),F_n(X))$. We will show a similar result for $\mathcal{S}F_n(f)$, under an additional assumption.

\begin{thrm}
  Let $X$ be a compact space and $f:X\to X$ be a homeomorphism with as least one wandering point $x_0$. If the sets $\alpha_{f}(x_0)$ and $\omega_{f}(x_0)$ are finite, then $h_{pol}(\mathcal{S}F_n(f))\geqslant n$, for all $n\geqslant2$.
\end{thrm}

\begin{proof}
We will use the same notation as in the proof of Theorem 6 in \cite{DKL}. Define $x_n=(f^n(x_0))$, $n\in\mathbb{Z}$ and $$Y:=\overline{O_f(x_o)}=\{x_n\mid n\in\mathbb{Z}\}\cup\alpha_{f}(x_0)\cup\omega_{f}(x_0).$$
The set $Y$ is compact and contains finitely many non-wandering points, i.e. the points in $\alpha_{f}(x_0)\cup\omega_{f}(x_0)$. Since $\mathcal{S}F_n(Y)$ is an $\mathcal{S}F_n(f)$-invariant subset of $\mathcal{S}F_n(X)$, we have
$$h_{pol}(\mathcal{S}F_n(f),\mathcal{S}F_n(X))\geqslant h_{pol}(\mathcal{S}F_n(f),\mathcal{S}F_n(Y))\stackrel{T.\ref{glavna}}=h_{pol}(F_n(f),F_n(Y)).$$
Now, following the aforementioned proof from \cite{DKL}, which is done by defining a semi-conjugation with a subshift, we have that $$h_{pol}(F_n(f),F_n(Y))\geqslant n.$$
\end{proof}

Again, in the same way, we can prove the previous result for the generalised suspension $\mathcal{S}F^{m}_n(X)$:

\begin{thrm}
  Let $X$ be a compact space and $f:X\to X$ be a homeomorphism with as least one wandering point $x_0$. If the sets $\alpha_{f}(x_0)$ and $\omega_{f}(x_0)$ are finite, then $h_{pol}(\mathcal{S}F^{m}_n(f))\geqslant n$, for all  $n>m\geqslant1$.
\end{thrm}

\begin{cor}
Let $f:[0,1]\to[0,1]$ be a homeomorphism satisfying $f\neq\mathrm{Id}$ and $f^2\neq\mathrm{Id}$. Then $h_{pol}(\mathcal{S}F_n(f))=h_{pol}(F_n(f))=n$, for all $n\geqslant2$.
\end{cor}

\begin{proof}
We can assume that $f$ is strictly increasing and that $f\neq\mathrm{Id}$ (otherwise, consider $f^2$ instead of $f$). Since in this case ${Fix}(f)={NW}(f)$, there exists at least one wandering point $x_0$. In order to deduce this result we only have to comment on the fact that $\alpha_{f}(x_0)$ and $\omega_{f}(x_0)$ are in fact one-element sets, since $\lim\limits_{n\to\pm\infty}f^n(x_0)$ exists.
\end{proof}

Not all orientation-preserving homeomorphisms $f:\mathbb{S}^1\to \mathbb{S}^1$ with at least one wandering point $\{x_0\}$ will meet the condition of finite $\alpha_{f}(x_0)$ and $\omega_{f}(x_0)$. If the rotational number $\rho(f)$ is irrational, then it is a known fact that the universal $\omega$-limit set is either $\mathbb{S}^1$ or a Cantor set. In the latter case, $f$ possesses a wandering point (see Proposition 6.8 in \cite{GT}) and $\omega_f$ is not finite.\\

\noindent{\bf Acknowledgements.} The author is very grateful to Jelena Kati\'c for many constructive discussions regarding non-wandering and chain-recurrent sets on hyperspaces.

 \end{document}